\documentclass[12pt]{amsart}

\usepackage{amsmath}
\usepackage{amssymb}
\usepackage{amsfonts}
\usepackage{amsthm}
\usepackage{enumerate}
\usepackage{hyperref}
\usepackage{paralist}
\usepackage{color}
\usepackage{thmtools}
\usepackage{changepage}
\usepackage{graphicx}
\numberwithin{equation}{section}

\newcommand{\Pu}{\mathrm{P}}
\newcommand{\NN}{\mathbb{N}}
\newcommand{\Df}{\mathfrak{D}}
\newcommand{\st}{\mathrm{stel}}
\newcommand{\df}{\mathfrak{d}}
\newcommand{\floor}[1]{\lfloor #1 \rfloor}

\newcommand{\ZZ}{\mathbb{Z}}

\DeclareMathOperator{\exc}{exc}
\DeclareMathOperator{\link}{link}
\DeclareMathOperator{\sd}{sd}

\theoremstyle{plain}
\newtheorem{thm}{Theorem}[section]

\newtheorem{lem}[thm]{Lemma}
\newtheorem{cor}[thm]{Corollary}

\newtheorem{ques}[thm]{Question}

\newtheorem{prob}[thm]{Problem}

\newtheorem{conj}[thm]{Conjecture}

\theoremstyle{definition}
\newtheorem{dfn}[thm]{Definition}

\newtheorem{exmp}[thm]{Example}

\newtheorem{rem}[thm]{Remark}

\newtheorem{dfns-rems}[thm]{Definitions and Remarks}
\newtheorem{notas-rems}[thm]{Notations and Remarks}
\newtheorem{exmps-rems}[thm]{Examples and Remarks}

\newlength\Thmindent
\renewcommand{\emph}{\textbf}

\begin{document}

\title[Local $h$-vectors of Quasi-Geometric and Barycentric Subdivisions]{Local $h$-vectors of Quasi-Geometric and Barycentric Subdivisions}

\author{Martina Juhnke-Kubitzke}

\address{Martina Juhnke-Kubitzke \& Richard Sieg, Universit\"at Osnabr\"uck, FB Mathematik/ Informatik, 49069 Osnabr\"uck, Germany}
\email{juhnke-kubitzke@uos.de, richard.sieg@uos.de}

\author{Satoshi Murai}
\address{
Satoshi Murai,
Department of Pure and Applied Mathematics,
Graduate School of Information Science and Technology,
Osaka University,
Suita, Osaka, 565-0871, Japan}
\email{s-murai@ist.osaka-u.ac.jp}

\author{Richard Sieg}

%

\begin{abstract}
In this paper, we answer two questions on local $h$-vectors, which were asked by Athanasiadis.
First, we characterize all possible local $h$-vectors of quasi-geometric subdivisions of a simplex.
Second, we prove that the local $\gamma$-vector of the barycentric subdivision of any CW-regular subdivision of a simplex is nonnegative. Along the way, we derive a new recurrence formula for the derangement polynomials.
\end{abstract}

 \thanks{The first and the third author were partially supported by the German Research Council DFG-GRK~1916. 
 The second author was partially supported by JSPS KAKENHI JP16K05102.}
 
\maketitle

\section{Introduction}
The classification of face numbers of (triangulated) spaces is an important and central topic not only in algebraic and geometric combinatorics but also in other fields, as e.g., commutative and homological algebra  and discrete, algebraic and toric geometry. The studied classes of spaces comprise abstract simplicial complexes, triangulated spheres and (pseudo)manifolds but also not necessarily simplicial objects such as (boundaries of) polytopes and Boolean cell complexes. In $1992$,  Stanley \cite{sta_sub} introduced the so-called \emph{local $h$-vector} of a topological subdivision of a $(d-1)$-dimensional simplex as a tool to study face numbers of subdivisions of simplicial complexes. His original motivation was the question, posed by Kalai and himself, if the (classical) $h$-vector increases under subdivision of a Cohen-Macaulay complex. Using local $h$-vectors Stanley could provide an affirmative answer to this question for so-called \emph{quasi-geometric subdivisions}. The crucial property of those, that he used, is that their local $h$-vectors are nonnegative; a property, which is no longer true if one considers arbitrary topological subdivisions. 

Complementing results by Chan \cite{chan1994subdivisions}, one of our main results shows that~--~except for the conditions that are true for any local $h$-vector~--~ nonnegativity already characterizes local $h$-vectors of quasi-geometric subdivisions entirely. More precisely, we show the following:

\begin{thm}\label{thm:charaQuasi}
Let $\ell=(\ell_0,\ldots,\ell_d)\in\ZZ^{d+1}$. The following conditions are equivalent:
\begin{itemize}
\item[(1)] There exists a quasi-geometric subdivision $\Gamma$ of the $(d-1)$-simplex such that the local $h$-vector of $\Gamma$ is equal to $\ell$. 
\item[(2)] $\ell$ is symmetric (i.e., $\ell_i=\ell_{d-i}$ for $0\leq i\leq d$), $\ell_0=0$ and $\ell_i\geq 0$ for $1\leq i\leq d-1$.
\end{itemize}
\end{thm}

We want to remark, that it already follows from \cite{sta_sub} that the local $h$-vector of any quasi-geometric subdivision satisfies (2). To prove Theorem \ref{thm:charaQuasi}, it therefore suffices to construct a quasi-geometric subdivision $\Gamma$ having a prescribed vector $\ell=(\ell_0,\ldots,\ell_d)\in\ZZ^{d+1}$, satisfying (2), as its local $h$-vector. For this, we will extend constructions by Chan \cite{chan1994subdivisions} who provided characterizations for  local $h$-vectors of regular and topological subdivisions.

As the local $h$-vector is symmetric \cite{sta_sub} it makes sense to define a \emph{local $\gamma$-vector}, which was introduced by Athanasiadis in \cite{athanasiadis2012flag} and  is defined in the same way as is the usual $\gamma$-vector for homology spheres. The central conjecture for local $\gamma$-vectors is the following, due to Athanasiadis \cite[Conjecture~5.4]{athanasiadis2012flag}.

\begin{conj} \label{conj:gamma}
The local $\gamma$-vector of a flag vertex-induced homology subdivision of a simplex is nonnegative.
\end{conj}
It can be shown that this conjecture is indeed a strengthening of Gal's conjecture for flag homology spheres \cite[Conjecture~2.1.7]{gal2005real} and, in particular, implies the Charney-Davis conjecture \cite{CharneyDavis}. 
Conjecture \ref{conj:gamma} is known to be true in small dimensions \cite{athanasiadis2012flag} and for various special classes of subdivisions, including barycentric, edgewise and cluster subdivisions of the simplex \cite{athanasiadis2012flag,athanasiadis2012cluster,athanasiadis2016edge} but besides it is still widely open. 
We add more evidence to it by showing the following:
\begin{thm}\label{thm:nonnegative}
Let $\Gamma$ be a CW-regular subdivision of a simplex. The local $\gamma$-vector of the barycentric subdivision $\sd(\Gamma)$ of $\Gamma$ is nonnegative. 
\end{thm}
The proof is based on an expression of the local $h$-vector which involves differences of $h$-vectors of restrictions of the subdivision and their boundary as well as derangement polynomials (Theorem \ref{thm:formulaoflocalh}). The nonnegativity is then concluded from a result on these differences by Ehrenborg and Karu \cite{ehrenborg_karu}.
As a byproduct of our proof we obtain a new recurrence formula for the derangement polynomials (Corollary 4.2).

We point out that the local $h$-vector, and therefore the local $\gamma$-vector, not only depends on the combinatorial type of $\sd(\Gamma)$ but also on the subdivision map. In Theorem \ref{thm:nonnegative}, we are considering the natural subdivision map of the barycentric subdivision, which will be explained in Section \ref{sect:CW}.

Theorems \ref{thm:charaQuasi} and \ref{thm:nonnegative} are motivated by questions asked by Athanasiadis \cite{athanasiadis2012flag,at_survey}. Indeed, Theorem \ref{thm:charaQuasi} partially solves \cite[Problem 2.11]{at_survey} and Theorem \ref{thm:nonnegative} answers \cite[Question 6.2]{athanasiadis2012flag}.

The paper is structured as follows. In Section~\ref{sec:preliminaries}, we provide the necessary background, including basic facts on simplicial complexes, topological subdivisions and more specifically barycentric subdivisions. Section~\ref{sec:char} is devoted to the characterization of local $h$-polynomials in the quasi-geometric case (Theorem \ref{thm:charaQuasi}). Finally, in Section~\ref{sec:bary} we prove Theorem \ref{thm:nonnegative}.

\section{Preliminaries}\label{sec:preliminaries}

First we provide some background material on simplicial complexes, their subdivisions and (local) $h$-vectors. 

\subsection{Simplicial complexes and their face numbers} 
Given a finite set $V$, a \emph{simplicial complex} $\Delta$ on $V$ is a family of subsets of $V$ which is closed under inclusion, i.e., $G\in \Delta$ and $F\subseteq G$ implies $F\in\Delta$. The elements of $\Delta$ are called \emph{faces} and the inclusion-maximal faces are called \emph{facets} of $\Delta$. The \emph{dimension} of a face $F\in\Delta$ is given by $|F|-1$ and the \emph{dimension} of $\Delta$ is the maximal dimension of its facets. If all facets of $\Delta$ have the same dimension, $\Delta$ is called \emph{pure}. We define the \emph{link} of a face $F\in\Delta$ to be
\begin{equation*}
\link_{\Delta}(F)=\left\{ G\in\Delta \mid G\cap F=\emptyset, G\cup F\in\Delta  \right\}.
\end{equation*}
The \emph{$f$-vector} $f(\Delta)=(f_{-1}(\Delta),f_0(\Delta),\ldots,f_{d-1}(\Delta))$ of a $(d-1)$-dimensional simplicial complex $\Delta$ encodes the number of $i$-dimensional faces ($-1\leq i\leq d-1$), i.e., 
\begin{equation*}
f_i(\Delta)=\left| \{F\in\Delta \mid \dim(F)=i \} \right| \quad \mbox{for } -1\leq i\leq d-1. 
\end{equation*}
Often it is more convenient to work with the \emph{$h$-vector} $h(\Delta)=(h_0(\Delta),\ldots,h_d(\Delta))$ of $\Delta$, which is defined by
\begin{equation*}
h_i(\Delta)=\sum_{j=0}^{i}(-1)^{i-j}\binom{d-j}{i-j}f_{j-1}(\Delta) \quad \mbox{for }0\leq i\leq d.
\end{equation*}
Regarding this vector as a sequence of coefficients yields the \emph{$h$-polynomial}
\begin{equation*}
h(\Delta,x)=\sum_{i=0}^{d}h_i(\Delta)x^i.
\end{equation*}
We refer the reader to \cite{sta_cca} for further background material.

\subsection{Subdivisions and local $h$-vectors}
The notion of local $h$-vectors goes back to Stanley \cite{sta_sub} and we recommend this article as a detailed reference (also see \cite{at_survey}).
Throughout this section, $V$ will always denote a nonempty finite set of cardinality $d$. A \emph{topological subdivision} of a simplicial complex $\Delta$ is a pair $(\Gamma, \sigma)$, where $\Gamma$ is a  simplicial complex and $\sigma$ is a map $\sigma:\Gamma\rightarrow\Delta$ such that, for any face $F \in \Delta$,
\begin{enumerate}[(i)]
\item[(1)] $\Gamma_F:=\sigma^{-1}(2^F)$ is a subcomplex of $\Gamma$ which is homeomorphic to a ball of dimension $\dim(F)$. $\Gamma_F$ is called the \emph{restriction} of $\Gamma$ to $F$.
\item[(2)] $\sigma^{-1}(F)$ consists of the interior faces of $\Gamma_{F}$.
\end{enumerate}
Following Stanley \cite{sta_sub}, we call the face $\sigma(G)\in\Delta$ the \emph{carrier} of $G\in\Gamma$. 
We also want to warn the reader not to confuse the notation $\Gamma_F$ with the induced subcomplex of $\Gamma$ on vertex set $F$, which might even consist just of the vertices in $F$. Also, note that it directly follows from condition (1) that $\sigma$ is inclusion-preserving, i.e., $\sigma(G)\subseteq \sigma(F)$ if $G\subseteq F$.  
In what follows, we often just write subdivision instead of topological subdivision if we are referring to a subdivision without additional properties, and we say that $\Gamma$ is a subdivision of $\Delta$ without referring to the map $\sigma$ if this one is clear from the context.

Let $\Gamma$ be a subdivision $(\Gamma,\sigma)$ of a simplicial complex $\Delta$. We say that $(\Gamma,\sigma)$ is 
\emph{quasi-geometric} if there do not exist $E\in\Gamma$ and $F\in\Delta$ with $\dim(F)<\dim(E)$ such that $\sigma(v)\subseteq F$ for any vertex $v$ of $E$. The subdivision $(\Gamma,\sigma)$ is \emph{vertex-induced} if for all faces $E\in\Gamma$ and $F\in\Delta$ such that every vertex of $E$ is a vertex of $\Gamma_F$, we have $E\in\Gamma_F$. Moreover, $(\Gamma,\sigma)$ is called \emph{geometric} if the subdivision $\Gamma$ admits a geometric realization that geometrically subdivides a geometric realization of $\Delta$. Finally, we say that $(\Gamma,\sigma)$ is \emph{regular} if the subdivision is induced by a weight function, i.e., it can be obtained via a projection of the lower hull of a polytope (see \cite[Definition~5.1]{sta_sub}). We have the following relations between those properties:
\begin{center}
$\{\mbox{topological subdivisions}\}\supsetneq \{\mbox{quasi-geometric subdivisions}\}\supsetneq$\\

 $\{\mbox{vertex-induced subdivisions}\}\supsetneq \{\mbox{geometric subdivisions}\}\supsetneq \{\mbox{regular subdivisions}\}$,
\end{center}
where all containments are strict. Figure~\ref{fig:subdiv_typ} shows examples of subdivisions of the $2$-simplex that are (a) regular, (b) geometric but not regular, (c) quasi-geometric but not vertex-induced and (d) not even quasi-geometric. A subdivision that is vertex-induced but not geometric is harder to depict but can be found in \cite{chan1994subdivisions}. 
\begin{figure}[h]
\includegraphics[scale=0.8]{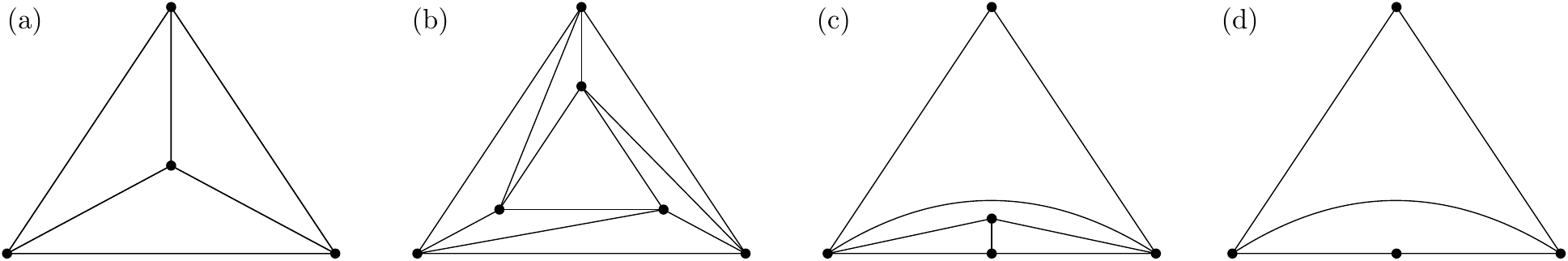}
\caption{Examples of subdivisions of the $2$-simplex.}
\label{fig:subdiv_typ}
\end{figure}

In 1992, Stanley introduced the local $h$-vector of a subdivision of a simplex as a tool to study the classical $h$-vector of the subdivision of a simplicial complex.

\begin{dfn}
Let $\Gamma$ be a subdivision of $2^V$. Then 
\begin{equation}\label{eq:local_h}
\ell_V(\Gamma,x)=\sum_{F\subseteq V}(-1)^{d-|F|}h(\Gamma_F,x)=\sum_{i=0}^{d}\ell_i(\Gamma)x^i
\end{equation} is called the \emph{local $h$-polynomial} of $\Gamma$ (with respect to $V$) and the vector $\ell_V(\Gamma)=(\ell_0(\Gamma),\ell_1(\Gamma),\ldots,\ell_d(\Gamma))$ is referred to as the \emph{local $h$-vector} of $\Gamma$ (with respect to $V$).
\end{dfn}
To make this article self-contained, we now summarize some of the most important properties of local $h$-vectors that will be used later on (see also \cite[Theorem~2.6]{at_survey}). For this, we recall that a sequence $(a_0,a_1,\ldots,a_m)\in \NN^{m+1}$ is called \emph{unimodal} if there exists $0\leq s\leq m$ such that $a_0\leq a_1\leq \cdots \leq a_s\geq a_{s+1}\geq \cdots \geq a_m$. 

\newpage

\begin{thm}[Stanley \cite{sta_sub}]\label{thm:sub_main}
\leavevmode
\begin{enumerate}
\item[(1)] Let $\Delta$ be a pure simplicial complex and let $\Gamma$ be a subdivision of $\Delta$. Then:
\begin{equation}\label{eq:local_h_formula}
h(\Gamma,x)=\sum_{F\in\Delta}\ell_F(\Gamma_F,x)h(\link_{\Delta}(F),x).
\end{equation}
\item[(2)]Let $V\neq \emptyset$ and let $\Gamma$ be a subdivision of $2^V$. Then:
\begin{enumerate}
\item[(a)] The local $h$-vector is symmetric, i.e., $\ell_i(\Gamma)=\ell_{d-i}(\Gamma)$ for $0\leq i\leq d$. Furthermore, $\ell_0(\Gamma)=0$ and $\ell_1(\Gamma)\geq0$.
\item[(b)] If $\Gamma$ is a quasi-geometric, then $\ell_i(\Gamma)\geq0$ for $0\leq i\leq d$.
\item[(c)] If $\Gamma$ is a regular, then $\ell_V(\Gamma)$ is unimodal.
\end{enumerate}
\end{enumerate}
\end{thm}
It was shown by Chan \cite{chan1994subdivisions} that the conditions in (2(a)) already characterize local $h$-vectors of topological subdivisions. Adding unimodality, one obtains the characterization of local $h$-vectors of regular subdivisions. We will complete this picture by showing in the next section that indeed every vector satisfying the conditions in (2(a)) and (2(b)) occurs as the local $h$-vector of a quasi-geometric subdivision.

In the following, let $\Gamma$ be a subdivision of $2^V$. 
As by Theorem \ref{thm:sub_main} (2(a)) $\ell_V(\Gamma)$ is symmetric, we can express the local $h$-polynomial $\ell_V(\Gamma,x)$ uniquely in the polynomial basis $\{x^k(1+x)^{d-2k}\mid 0\leq k\leq \lfloor d/2\rfloor\}$, i.e., 
\begin{equation*}
\ell_V(\Gamma,x)=\sum_{k=0}^{\lfloor d/2\rfloor}\xi_k(\Gamma) x^k(1+x)^{d-2k},
\end{equation*}
where $\xi_k(\Gamma)\in\ZZ$ are uniquely determined. The sequence $\xi_V(\Gamma)=(\xi_0(\Gamma),\ldots,\xi_{\lfloor d/2\rfloor}(\Gamma))$ is called the \emph{local $\gamma$-vector} of $\Gamma$ (with respect to $V$). 
The local $\gamma$-vector is known to be nonnegative for flag vertex-induced subdivisions in dimension $\leq 3$ \cite{athanasiadis2012flag} and for special classes of subdivisions including barycentric, edgewise and cluster subdivisions of the simplex\cite{athanasiadis2012flag,athanasiadis2012cluster}. 

More generally, any symmetric polynomial $p(x)$ with center of symmetry $n/2$ is called \emph{$\gamma$-nonnegative} if the coefficients $\gamma_k$ given by
\begin{equation*}
p(x)=\sum_{k=0}^{\floor{n/2}}\gamma_k x^k(1+x)^{n-2k}.
\end{equation*}
are nonnegative.

\subsection{CW-regular subdivisions}\label{sect:CW}
The definition of topological subdivisions can be naturally extended to regular CW-complexes \cite[\S 7]{sta_sub}.
For a regular CW-complex $\Gamma$,
we write $P(\Gamma)$ for the face poset of $\Gamma$.
A \emph{CW-regular} subdivision of a simplicial complex $\Delta$ is a pair $(\Gamma,\sigma)$, where $\Gamma$ is a regular CW-complex and $\sigma : P(\Gamma) \to \Delta$ is a map satisfying the conditions (1) and (2) of topological subdivisions.

Given a regular CW-complex $\Gamma$, its barycentric subdivision $\sd(\Gamma)$ is the simplicial complex, whose $i$-dimensional faces are given by chains
\[
\tau_0 \lneq \tau_1 \lneq \cdots \lneq \tau_i,
\]
where $\tau_j \in P(\Gamma)$ is a non-empty face of $\Gamma$ ($0\leq j\leq i$).
It is well-known that $\Gamma$ and $\sd(\Gamma)$ are homeomorphic.
Then, for a CW-regular subdivision $(\Gamma,\sigma)$ of a simplicial complex $\Delta$,
$\sd(\Gamma)$ can be naturally considered as a subdivision of $\Delta$ by the map 
\begin{equation*}
\sigma'(\{\tau_0,\tau_1,\dots,\tau_i\})
:=\sigma(\max\{\tau_0,\tau_1,\dots,\tau_i\}),
\end{equation*}
since
\begin{equation*}
(\sigma')^{-1}(2^F)=\big\{ \{\tau_0,\tau_1,\dots,\tau_i\} \in \sd(\Gamma): \sigma(\max\{\tau_0,\tau_1,\dots,\tau_i\}) \subseteq F\big\}=\sd (\Gamma_F),
\end{equation*}
where $\Gamma_F=\{ \tau \in P(\Gamma): \sigma(\tau) \subseteq F\}$.
When we consider the local $h$-polynomials of $\sd(\Gamma)$, we always consider the local $h$-polynomial using the above map $\sigma'$.

\subsection{Derangement polynomial}

Let $V$ be a set of cardinality $d$. It is easy to see that the barycentric subdivision $\sd(2^V)$ of a $(d-1)$-simplex $2^V$ is a special instance of a regular subdivision. Its $h$-vector $h(\sd(2^V))=(h_0(\sd(2^V)),\ldots,h_d(\sd(2^V)))$ is given by 
\begin{equation*}
h_i(\sd(2^V))=A(d,i),
\end{equation*}
where $A(d,i)$ are the Eulerian numbers, counting the number of permutations in the symmetric group $S_d$ on $d$ elements with exactly $i$ descents. Recall that a \emph{descent} of a permutation $\pi\in S_d$ is an index $1\leq k\leq d-1$ such that $\pi(k)>\pi(k+1)$. 
Similarly, there is an expression of the local $h$-vector $\ell_V(\sd(2^V))$ of $\sd(2^V)$ involving permutation statistics. An \emph{excedance} of a permutation $\pi$ is an index $1\leq i\leq d$ such that $\pi(i)>i$. We write $\exc(\pi)$ for the number of excedances of $\pi$, i.e.,  
\begin{equation*}
\exc(\pi)=\left|\{ i\in[d] \mid \pi(i)>i \}  \right|.
\end{equation*} 
A permutation $\pi\in S_d$ is called a \emph{derangement} if it does not have any fixed point. We denote by $\Df_d$ the set of derangements in $S_d$. The \emph{derangement polynomial} of order $d$ is defined by
\begin{equation*}
\df_d(x)=\sum_{\pi\in \Df_d}x^{\exc(\pi)}.
\end{equation*}
It is convenient to set $\df_0(x)=1$. 
These polynomials were first studied by Brenti in \cite{bre90}. 
It is not hard to see (see also \cite[Proposition~2.4.]{sta_sub}) that the local $h$-polynomial of $\sd(2^V)$ is given by 
\begin{equation}\label{eq:local_h_bary}
\ell_V(\sd(2^V),x)=\df_d(x).
\end{equation}
It was shown in \cite{zhang1995q} that derangement polynomials, and thus the local $h$-polynomial of $\sd(2^V)$ is $\gamma$-nonnegative. 

\section{Characterization of local $h$-polynomials}\label{sec:char}
In this section, we provide a characterization of local $h$-vectors of quasi-geometric subdivisions. This complements work by Chan \cite{chan1994subdivisions}, who showed that local $h$-vectors of topological and regular subdivisions are completely characterized by their properties in  Theorem \ref{thm:sub_main} (2).  In particular, we prove Theorem \ref{thm:charaQuasi} by extending her main idea. 

As local $h$-vectors of quasi-geometric subdivisions are known to be symmetric and nonnegative with first entry equal to $0$ (see Theorem \ref{thm:sub_main} (2)), we only need to show that the conditions in (2b) are also sufficient. Let $\ell=(\ell_0,\ldots,\ell_d)\in\ZZ^{d+1}$ be fixed and assume that $\ell$ satisfies the conditions in Theorem \ref{thm:charaQuasi} (2). We will explicitly construct a quasi-geometric subdivision $\Gamma$ of a $(d-1)$-simplex $2^V$ with $\ell_V(\Gamma)=\ell$. The basic idea is to find operations on a subdivision $\Delta$ of a simplex that preserve quasi-geometricity, change the local $h$-vector of $\Delta$ in a prescribed way and such that $\ell$ can be realized as local $h$-vector of a subdivision obtained by successively applying these operations. 

In \cite{chan1994subdivisions} Chan already provided three operations that suffice to construct all local $h$-vectors of arbitrary topological subdivisions. Though one of these does not necessarily preserve quasi-geometricity,  we recall her constructions and their effects on the local $h$-vector, since we will use them in what follows.\\
Let $V$ be a set with $|V|=d$ and let $(\Gamma,\sigma)$ be a subdivision of the $(d-1)$-simplex $2^V$.

\begin{enumerate}
\item[(O1)] Let $(\st_{\Gamma}(F),\sigma')$ be obtained from $\Gamma$ by stellar subdivision of a facet $F$ of $\Gamma$, where $\sigma'(z)=\sigma(F)$ for the new vertex $z$. Then:
\begin{equation*}
 \ell_V(\st_{\Gamma}(F))=\ell_V(\Gamma)+(0,1,\ldots,1,0).
\end{equation*}
\item[(O2)] Let $d\geq 4$ and $G$ be a $(d-2)$-dimensional face of $\Gamma$ with $(d-2)$-dimensional carrier $\sigma(G)$. Let $(\mathrm{P}_\Gamma(G),\sigma')$ be the subdivision of $2^V$ obtained from $\Gamma$ by adding a new vertex $w$ with carrier $\sigma'(w)=\sigma(G)$ and one new facet $G\cup\{w\}$ (with carrier $V$). (Note that $\sigma'(G)=V$.) Then:
\begin{equation*}
\ell_V(\Pu_{\Gamma}(G))=\ell_V(\Gamma)+(0,0,-1,\ldots,-1,0,0).
\end{equation*}
We will say that $\Pu_\Gamma(G)$ is obtained from $\Gamma$ by \emph{pushing} $G$ into the interior. 
\item[(O3)] Let $\Omega$ be the subdivision of a $1$-simplex $2^{\{d+1,d+2\}}$ into two edges. Then $\Gamma^{*1}:=\Gamma\ast\Omega$ is a subdivision of the $(d+1)$-simplex $2^{V\cup\{d+1,d+2\}}$ with 
\begin{equation*}
\ell_{V'}(\Gamma^{*1})=(0,\ell_V(\Gamma),0).
\end{equation*}
\end{enumerate}
The operations are depicted in Figure~\ref{fig:operations}.
\begin{figure}[h]
\includegraphics[scale=0.7]{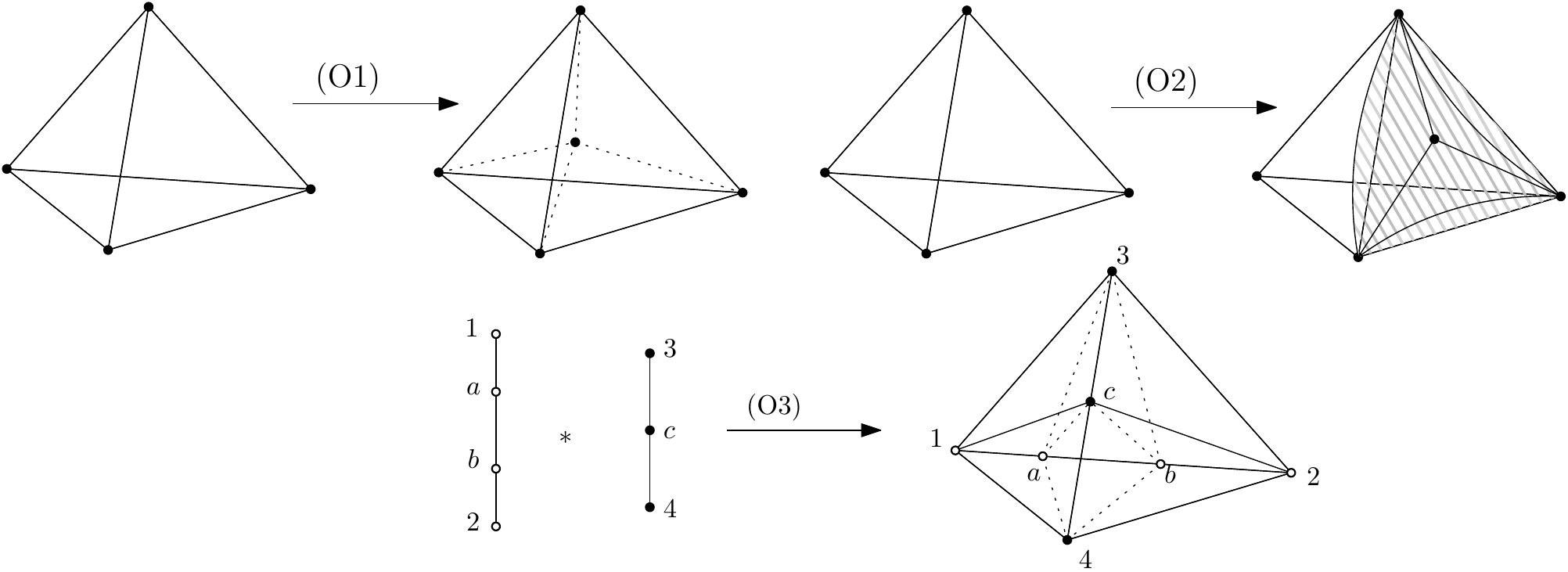}
\caption{Constructions for subdivisions.}
\label{fig:operations}
\end{figure}
Chan showed that~--~starting from a $2$- or $3$-simplex~--~these constructions suffice to generate any symmetric vector $\ell\in\ZZ^{d+1}$ with $\ell_0=0$ and $\ell_1\geq0$. Moreover, both, the stellar subdivision (O1) and the join operation (O3), maintain regularity of a subdivision and any symmetric and unimodal vector $\ell\in\ZZ^{d+1}$ with $\ell_0=0$ can be constructed by their successive application \cite{chan1994subdivisions}.

It is straight forward to show that that stellar subdivision (O1) and the join operation (O3) behave well with respect to quasi-geometricity.

\begin{lem}\label{lem:operations_gq}
Let $\Gamma$ be a quasi-geometric subdivision of $2^V$. 
 Then:
\begin{enumerate}
\item[(1)] If $F$ is a facet of $\Gamma$, then $\st_{\Gamma}(F)$ is a quasi-geometric subdivision of $2^V$.
\item[(2)] $\Gamma^{*1}$ is a quasi-geometric subdivision of $2^{V\cup\{d+1,d+2\}}$.
\end{enumerate}
\end{lem}

Even though, $\Pu_\Gamma(G)$ (if defined) might not be quasi-geometric (even if $\Gamma$ is), the next lemma shows that this obstruction can be ``repaired'' with just  one additional stellar subdivision.
\begin{lem}\label{lem:O4}
Let $\Gamma$ be a quasi-geometric subdivision of $2^V$ and let $d=|V|\geq4$. Let  $G$ be a $(d-2)$-dimensional face of $\Gamma$ with $(d-2)$-dimensional carrier. Let $w$ be the new vertex of $\Pu_\Gamma(G)$. Then the subdivision $\st_{\Pu_{\Gamma}(G)}(G\cup\{w\})$ obtained by first pushing $G$ into the interior of $\Gamma$ and then stellarly subdividing the new facet $G\cup\{w\}$ is a quasi-geometric subdivision of $2^V$. Moreover,
\begin{equation*}
\ell_V(\st_{\Pu_{\Gamma}(G)}(G\cup\{w\}))=\ell_V(\Gamma)+(0,1,0,\ldots,0,1,0).
\end{equation*}
\end{lem}

Figure~\ref{fig:quasigeom} shows $\st_{\Pu_{\Gamma}(G)}(G\cup\{w\})$ in the case that $\Gamma$ is just a $3$-simplex.

\begin{proof}
To simplify notation, we set $\Gamma':=\st_{\Pu_{\Gamma}(G)}(G\cup\{w\})$. 
The claim about the local $h$-vector follows immediately from the definition of the used operations (see also \cite{chan1994subdivisions}). 

It remains to verify that $\Gamma'$ is a quasi-geometric subdivision of $2^V$. We denote by $\sigma_1:\Gamma\rightarrow 2^V$ the map corresponding to the subdivision $\Gamma$ of $2^V$ and by $\sigma_2:\Gamma'\rightarrow \Gamma$ the subdivision map of $\Gamma'$ (as a subdivision of $\Gamma$). The subdivision map of $\Gamma'$ as a subdivision of $2^V$ is then given by $\sigma:=\sigma_2\circ\sigma_1$. Let $z$ be the newly added vertex when applying stellar subdivision to $G\cup\{w\}$.


We first note that by definition of $\Gamma'$ and $\sigma$ we have
\begin{equation*}
\sigma(E)=\sigma_1(E),\quad\forall\,E\in\Gamma\cap\Gamma'\setminus\{G\},\quad \sigma(w)=\sigma_1(G) \quad \mbox{and}\quad \sigma(z)=V.
\end{equation*}
Given a face $E\in \Gamma'$, we need to show that the following condition, referred to as condition (QG) in the sequel, is satisfied:\\

\begin{itemize}
\item[(QG)] For all $F\subseteq V$ such that $\sigma'(v)\subseteq F$  for all $v\in E$,  it holds that $\dim F\geq \dim E$.\\
\end{itemize}

Let $E\in \Gamma\cap\Gamma'\setminus \{G\}$. In this case, we have $\sigma(E)=\sigma_1(E)$ and $\sigma(u)=\sigma_1(u)$ for all $u\in E$. As $\Gamma$ is quasi-geometric, it follows, that $E$ satisfies condition (QG).

Similarly, we have $\sigma(u)=\sigma_1(u)$ for all $u\in G$ and as $\Gamma$ is quasi-geometric,  condition (QG) holds for $G$. 

It remains to consider faces $E\in \Gamma'\setminus \Gamma$. First assume $z\in E$. As $\sigma(z)=V$ by construction and $|V|=d$, those faces satisfy condition (QG). Suppose that $z\notin E$. As $E\notin \Gamma$, we must have that $w\in E$. Since $z\notin E$, we can further conclude that $\dim E\leq d-2$. (Indeed, any facet containing $w$ also contains $z$.) Combining this with the fact that $\sigma(w)=\sigma_1(G)$ is of dimension $d-2$ (by assumption) we get that $E$ meets condition (QG). The claim follows.
\end{proof}

We can finally provide the proof of Theorem \ref{thm:charaQuasi}, i.e., the desired characterization of local $h$-vectors of quasi-geometric subdivisions.\\

\begin{proof}[Proof of Theorem~\ref{thm:charaQuasi}]
The ``only if''-part follows directly from Theorem~\ref{thm:sub_main} (b). 
We now show the ``if''-part. If $d$ is even we start with a $2$-simplex and take $\ell_{\frac{d}{2}}$ times the stellar subdivision (O1) of a facet. Similarly, if $d$ is odd, we start with the $3$-simplex and take the stellar subdivision (O1) of a facet $\ell_{\frac{d-1}{2}}$ times.
In the next step, we apply once (O3) and then $\ell_{\lfloor\frac{d}{2}\rfloor -1}$ times the operation defined in Lemma~\ref{lem:O4}. We continue in this way and, by Lemmas~\ref{lem:operations_gq} and \ref{lem:O4}, this yields a quasi-geometric subdivision $\Gamma$ of the $(d-1)$-simplex whose local $h$-vector is equal to $\ell$.
\end{proof}

Chan and Stanley originally conjectured that all local $h$-vectors of quasi-geometric subdivisions are unimodal. However,  Athanasiadis disproved this conjectured by providing a counterexample to it (\cite[Example~3.4]{athanasiadis2012flag} and \cite{at_survey}). This example is obtained by applying the operation defined in Lemma \ref{lem:O4} to the $3$-simplex (see Figure~\ref{fig:quasigeom}) and has local $h$-vector $(0,1,0,1,0)$.
\begin{figure}[h]
\includegraphics[scale=0.5]{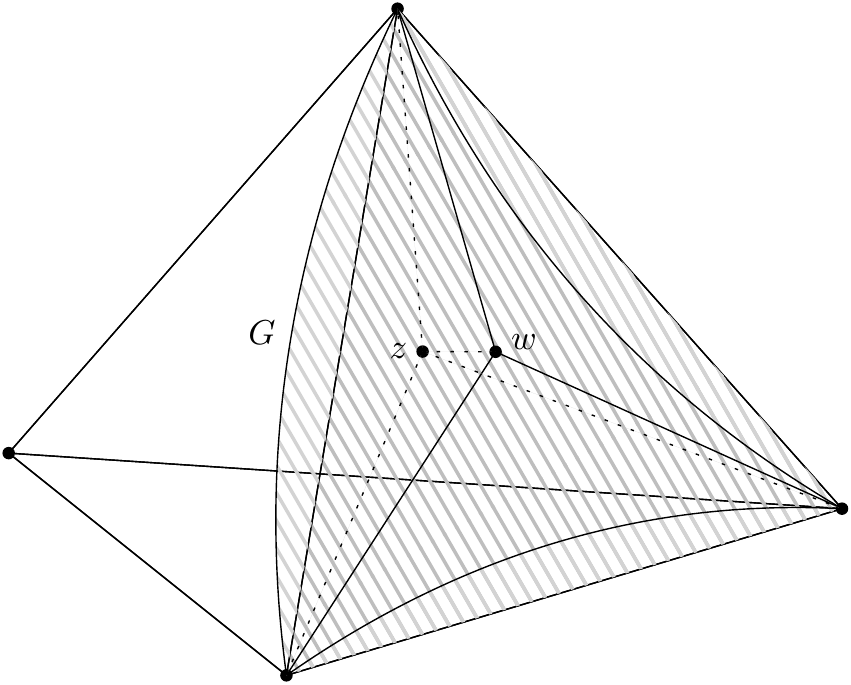}
\caption{Quasi-geometric subdivision with non-unimodal $h$-polynomial.}
\label{fig:quasigeom}
\end{figure}
 Nevertheless, no geometric or even just vertex-induced subdivisions of the $(d-1)$-simplex are known whose local $h$-vector is not unimodal. Already, Athanasiadis in \cite[Question 3.5]{athanasiadis2012flag} asked if such examples exist or if all vertex-induced subdivision of the $(d-1)$-simplex have unimodal local $h$-vector. Based on a lot of experiments and a great vain effort to construct counterexamples we are inclined to believe that the latter is indeed the case.

\section{Local $\gamma$-vectors of barycentric subdivisions}\label{sec:bary}

In this section, we  provide the proof of Theorem \ref{thm:nonnegative}, i.e., we show that the local $\gamma$-vector of the barycentric subdivision of any CW-regular subdivision of a simplex is nonnegative. This answers Question 6.2 in \cite{athanasiadis2012flag} in the affirmative.\\

As an example of the type of subdivision we are interested in, Figure~\ref{fig:stellar_bary} depicts the barycentric subdivision of the stellar subdivision of the $2$-simplex.
\begin{figure}[h]
\includegraphics[scale=0.8]{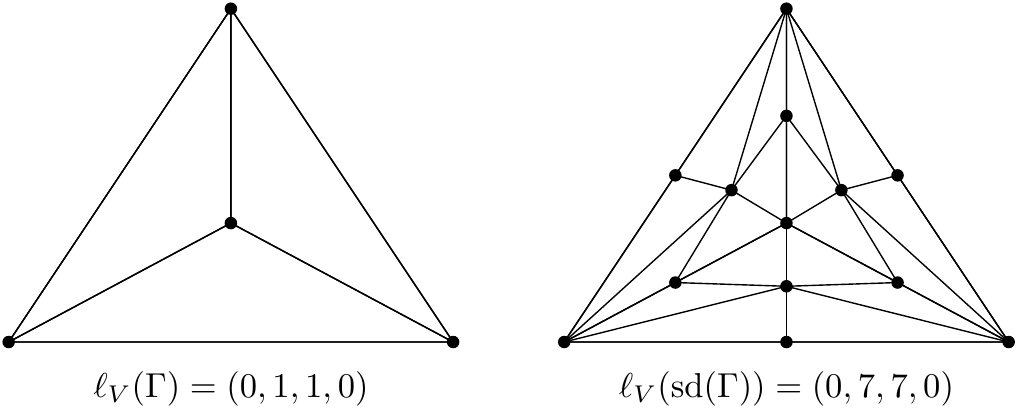}
\caption{Barycentric subdivision of the stellar subdivision.}
\label{fig:stellar_bary}
\end{figure}

\begin{lem}\label{lem:local_h}
Let $V$ be a finite set with $|V|=d$ and let $\Delta$ be a subdivision of the simplex $2^V$. The local $h$-polynomial of $\Delta$ can be written as
\begin{align*}
\ell_V(\Delta,x)=h(\Delta,x)-h(\partial \Delta,x)+\sum_{F\subsetneq V}\ell_F(\Delta_F,x)(x+\cdots+x^{d-|F|-1}).
\end{align*}
Here, $\partial \Delta$ denotes the boundary of $\Delta$. 
\end{lem}
\begin{proof}
First, note that for any $F\subsetneq V$ the link of $F$ in $2^V$ respectively in $\partial(2^V)$ is a $(|V|-|F|-1)$-simplex respectively its boundary. Hence, 
\begin{equation*}
h(\link_{2^V}(F),x)=1 \quad \mbox{and} \quad h(\link_{\partial(2^V)}(F),x)=1+x+\cdots+x^{d-|F|-1}.
\end{equation*} 

Applying (\ref{eq:local_h_formula}) to $\Delta$, we obtain
\begin{align}
h(\Delta,x)&=\sum_{F\subseteq V}\ell_F(\Delta_F,x)h(\link_{2^V}(F),x)\\
&=\ell_V(\Delta,x)+\sum_{F\subsetneq V}\ell_F(\Delta_F,x).\nonumber
\end{align}
Similarly, viewing $\partial \Delta$ as a subdivision of $\partial(2^V)$ and using that $(\partial \Delta)_F=\Delta_F$ for $F \subsetneq V$, 
the $h$-polynomial of $\partial\Delta$ can be written in the following way:
\begin{align}
h(\partial \Delta,x)&=\sum_{F\subsetneq V}\ell_F(\Delta_F,x)h(\link_{\partial (2^V)}(F),x)\label{eq:boundary}\\
&=\sum_{F\subsetneq V}\ell_F(\Delta_F,x)(1+x+\cdots+x^{d-|F|-1}).\nonumber
\end{align}
Subtracting (4.2) from (4.1) yields
\begin{equation*}
\ell_V(\Delta,x)=h(\Delta,x)-h(\partial\Delta,x)+\sum_{F\subsetneq V}\ell_F(\Delta_F,x)(x+\cdots+x^{d-|F|-1}),
\end{equation*}
as desired.
\end{proof}
Using \eqref{eq:local_h_bary} and the fact that $\sd(2^V)$ has the same $h$-polynomial as its boundary 
(since it is just the cone over it), we get the following  recurrence formula for the derangement polynomials as a special case of Lemma \ref{lem:local_h} when $\Delta=\sd(2^V)$. We recall that $\df_0(x)=1$ by definition.

\begin{cor}\label{cor:dn}
For every $d\in\NN$,
\begin{equation*}
\df_d(x)=\sum_{k=0}^{d-2}\binom{d}{k}\df_k(x)(x+\cdots+x^{d-1-k}).
\end{equation*}
\end{cor}
To the best of our knowledge this formula seems to be new, as we could not find it in the literature. Note that it directly implies the unimodality of $\df_n(x)$.

\begin{rem}
We also found a purely combinatorial proof of Corollary~\ref{cor:dn} using a similar recurrence formula for the Eulerian polynomials. This proof will appear in the PhD thesis of the third author.
\end{rem}
The next theorem is crucial in the proof of Theorem \ref{thm:nonnegative}.
\begin{thm}\label{thm:formulaoflocalh}
Let $V\neq \emptyset$ be a finite set and let $\Delta$ be a subdivision of the simplex $2^V$. The local $h$-polynomial of $\Delta$ can be written as
\begin{equation}\label{eq:nonnegative}
\ell_V(\Delta,x)=\sum_{F\subseteq V}\left[h(\Delta_F,x)-h(\partial(\Delta_F),x)\right]\cdot \df_{|V\setminus F|}(x).
\end{equation}
\end{thm}
\begin{proof}
To simplify notation, we set
\begin{align*}
h_F(x)=h(\Delta_F,x)-h(\partial(\Delta_F),x).
\end{align*}
We show the claim by induction on $|V|$. If $|V|=1$, both sides in \eqref{eq:nonnegative} are equal to $0$ and the claim is trivially true. 

Assume $|V|\geq 2$. By Lemma~\ref{lem:local_h} and the induction hypothesis, we have
\begin{align*}
\ell_V(\Delta,x)=&h_V(x)+\sum_{F\subsetneq V}\left[\sum_{G\subseteq F}h_G(x)\cdot \df_{|F\setminus G|}(x) \right](x+\cdots+x^{|V|-|F|-1})\\
=&h_V(x)+\sum_{G\subsetneq V}h_G(x)\left[\sum_{G\subseteq F\subsetneq V} \df_{|F\setminus G|}(x) (x+\cdots+x^{|V|-|F|-1})\right]\\
=&h_V(x)+\sum_{G\subsetneq V}h_G(x)\left[\sum_{F\subsetneq V\setminus G}\df_{|F|}(x)(x+\cdots+x^{|V|-|G|-1-|F|})\right]\\
=&h_V(x)+\sum_{G\subsetneq V}h_G(x)\cdot \df_{|V\setminus G|}(x),
\end{align*}
where the last equality follows from Corollary~\ref{cor:dn}.
\end{proof}

\newpage

\begin{exmp}
We illustrate Theorem~\ref{thm:formulaoflocalh} using two examples. 
\begin{itemize}
\item[(1)] Let $\Delta=2^V$ be the trivial subdivision of the simplex $2^V$ with $|V|=d\geq 1$. Then $\ell_V(\Delta,x)=0$ and
\begin{equation*}
h(\Delta_F,x)-h(\partial(\Delta_F),x)= -x-x^2-\dots-x^{|F|-1} \quad \mbox{ for every }\quad \emptyset \ne F\subseteq V.
\end{equation*} 
However, $h(\Delta_\emptyset,x)-h(\partial(\Delta_\emptyset),x)=1$ and all negative terms on the right-hand side of \eqref{eq:nonnegative} cancel out. We retrieve the recurrence formula from Corollary~\ref{cor:dn}:
\begin{align*}
0=\ell_V(\Delta,x)&=\df_d(x)+\sum_{\emptyset \ne F\subseteq V}(-x-\dots-x^{|F|-1})\df_{|V\setminus F|}(x)\\
&=\df_d(x)+\sum_{k=2}^{d}\binom{d}{k}(-x-\dots-x^{k-1})\df_{d-k}(x)\\
&=\df_d(x)-\sum_{k=0}^{d-2}\binom{d}{k}(x+\dots+x^{d-k-1})\df_{k}(x).
\end{align*}
\item[(2)] Let $\Delta$ be the barycentric subdivision of the stellar subdivision of the $2$-simplex, as depicted in Figure~\ref{fig:stellar_bary}.
Then $\df_0(x)=1, \df_1(x)=0$ and
\begin{equation*}
h(\Delta,x)-h(\partial\Delta,x)=(1+10x+7x^2)-(1+4x+x^2)=6x+6x^2.
\end{equation*} 
The right-hand side of \eqref{eq:nonnegative} is 
\begin{align*}
&\sum_{F\subseteq V}\left[h(\Delta_F,x)-h(\partial(\Delta_F),x)\right]\cdot \df_{|V\setminus F|}(x)\\
=&6x+6x^2+\sum_{F\subseteq V, |F|\leq 1}\left[h(\Delta_F,x)-h(\partial(\Delta_F),x)\right]\cdot \df_{|V\setminus F|}(x)\\
=&6x+6x^2+\df_3(x)=7x+7x^2,
\end{align*}
which is indeed equal to the local $h$-polynomial of $\Delta$ (see Figure~\ref{fig:stellar_bary}).
\end{itemize}
\end{exmp}
We now prove Theorem \ref{thm:nonnegative}.
Let $\Gamma$ be a CW-regular subdivision of a simplex $2^V$.
Then, by applying the special case of Theorem \ref{thm:formulaoflocalh} when $\Delta=\sd(\Gamma)$, we obtain
\[
\ell_V(\sd(\Gamma),x)=\sum_{F \subseteq V} \big[ h(\sd(\Gamma_F),x) - h(\partial(\sd(\Gamma_F)),x) \big]
\cdot \df_{|V\setminus F|}(x).
\]
Since we already know that $\df_k(x)$ is $\gamma$-nonnegative and since the product of two $\gamma$-nonnegative polynomials is $\gamma$-nonnegative,
the next result due to Ehrenborg and Karu \cite{ehrenborg_karu}
completes the proof of Theorem \ref{thm:nonnegative}.

\begin{thm}[Ehrenborg--Karu]\label{cdindex}
Let $\Gamma$ be a regular CW-complex which is homeomorphic to a ball. Then $h(\sd(\Gamma),x)-h(\partial(\sd(\Gamma)),x)$ is $\gamma$-nonnegative.
\end{thm}

Theorem \ref{cdindex} is an immediate consequence of \cite[Theorem 2.5]{ehrenborg_karu}, but since this result is written in the language of $\mathbf{cd}$-indices, we explain how Theorem \ref{cdindex} can be deduced from it.
Before the proof, we recall flag $h$-numbers and $\mathbf a \mathbf b$-indices.
Let $\Gamma$ be a regular CW-complex of dimension $d-1$.
An \textbf{$S$-chain} of $\Gamma$, where $S \subseteq [d]=\{1,2,\dots,d\}$, is a chain of $\Gamma$
\[
\tau_0 \lneq \tau_1 \lneq \cdots \lneq \tau_i
\]
with $S=\{ \dim \tau_0+1,\dots,\dim \tau_i+1\}$.
Let $f_S(\Gamma)$ be the number of $S$-chains of $\Gamma$.
Then, for $S \subseteq [d]$, we define $h_S(\Gamma)$ by $h_S(\Gamma)= \sum_{T \subseteq S} (-1)^{|S|-|T|} f_T(\Gamma)$.
Note that one has 
\begin{equation*}
f_i(\sd(\Gamma))=\sum_{\substack{S\subseteq [d] \\ |S|=i+1}} f_S(\Gamma)\quad \text{and} \quad h_i(\sd(\Gamma))=\sum_{\substack{S\subseteq [d] \\ |S|=i}} h_S(\Gamma).
\end{equation*}
Let $\mathbb Z\langle \mathbf a, \mathbf b\rangle$ and
$\mathbb Z\langle \mathbf c, \mathbf d\rangle$ be noncommutative polynomial rings, where $\mathbf {a},\mathbf {b},\mathbf {c},\mathbf {d}$ are variables with $\deg \mathbf a=\deg \mathbf b = \deg \mathbf c=1$ and $\deg \mathbf d=2$.
We say that a polynomial $f \in \mathbb Z \langle \mathbf c, \mathbf d\rangle$ is \emph{nonnegative} if all coefficients of monomials in $f$ are nonnegative.
For $S \subseteq [d]$, we define the noncommutative monomial $u_S=u_1u_2\cdots u_d \in \mathbb Z \langle \mathbf a, \mathbf b \rangle$ by
$u_i=\mathbf a$ if $i \not \in S$ and $u_i=\mathbf b$ if $i \in S$.
The homogeneous polynomial 
\[
\Psi_\Gamma(\mathbf a, \mathbf b) = \sum_{S \subseteq [d]} h_S(\Gamma) u_S
\]
is called the \emph{$\mathbf {ab}$-index} of $\Gamma$.
Note that by substituting $\mathbf a=1$ to $\Psi_\Gamma(\mathbf a ,\mathbf b)$
we obtain the $h$-polynomial of $\sd(\Gamma)$, that is,
$\Psi_\Gamma (1,\mathbf b)= \sum_{i=0}^d h_i(\sd(\Gamma)) \mathbf b^i$.

\begin{proof}[Proof of Theorem \ref{cdindex}]
Let $\Gamma$ be a regular CW-complex which is homeomorphic to a $(d-1)$-dimensional ball.
Then the face poset of $\Gamma$ is near-Gorenstein* in the sense of \cite[Definition 2.2]{ehrenborg_karu} (indeed, the near-Gorenstein* property is an abstraction of being a ball), and \cite[Theorem 2.5]{ehrenborg_karu} says that there is a nonnegative homogeneous polynomial $\Phi(\mathbf c, \mathbf d) \in \mathbb Z\langle \mathbf c, \mathbf d\rangle$ of degree $d$ such that
\[
\Phi(\mathbf a +\mathbf b, \mathbf a\mathbf b + \mathbf b \mathbf a)= \Psi_\Gamma(\mathbf a, \mathbf b) - \Psi_{\partial \Gamma}(\mathbf a, \mathbf b) \cdot \mathbf a.\]
By substituting $\mathbf a=1$ in the above equation,
we see that
\begin{align}
\label{4-A}
\Phi(1+ \mathbf b, 2 \mathbf b)= \Psi_\Gamma (1, \mathbf b)- \Psi_{\partial \Gamma} (1,\mathbf b)
= \sum_{i=0}^d \big( h_i(\sd(\Gamma))-h_i (\partial (\sd(\Gamma)) ) \big) \mathbf b^i
\end{align}
coincides with the polynomial $h(\sd(\Gamma),\mathbf b)-h(\partial (\sd(\Gamma)),\mathbf b)$.
On the other hand, since $\Phi(\mathbf c,\mathbf d)$ is homogeneous and nonnegative, there exist nonnegative integers $\gamma_0,\gamma_1,\gamma_2,\dots$ such that
\begin{align}
\label{4-B}
\Phi(1+ \mathbf b, 2 \mathbf b)=
\sum_{k=0}^{\lfloor d/2 \rfloor}
\gamma_k (1+ \mathbf b)^{d-2k} (2 \mathbf b)^k.
\end{align}
The two equations \eqref{4-A} and \eqref{4-B} guarantee the $\gamma$-nonnegativity of $h(\sd(\Gamma),\mathbf b)-h(\partial (\sd(\Gamma)),\mathbf b)$.
\end{proof}

As $\gamma$-nonnegativity implies unimodality, we obtain the following immediate corollary.

\begin{cor}
Let $\Gamma$ be a CW-regular subdivision of a simplex $2^V$. Then $\ell_V(\sd(\Gamma))$ is unimodal.
\end{cor}

\begin{rem}
The polynomials $h(\Delta_F,x)-h(\partial (\Delta_F),x)$ in Theorem \ref{thm:formulaoflocalh} are always symmetric.
Indeed, if $\Delta$ is a simplicial complex which is homeomorphic to a $(d-1)$-ball with $h(\Delta)=(h_0,h_1,\dots,h_d)$, then $h_d=0$ and $h_i(\partial \Delta)= \sum_{j=0}^i (h_j-h_{d-j})$ for all $i$
(see \cite[p.\ 137]{sta_cca}). Hence
\begin{align}
\label{rem:last}
h_i(\Delta)-h_i(\partial \Delta)= \sum_{j=0}^{i-1} (h_j-h_{d-j-1}),
\end{align}
and therefore $h_i(\Delta)-h_i(\partial \Delta)=h_{d-i}(\Delta)-h_{d-i}(\partial \Delta)$
for all $i=0,1,\dots,d$.
Moreover, \eqref{rem:last} says that if $h_j \geq h_{d-j-1}$ holds for $j < d/2$, then $h(\Delta,x)-h(\partial \Delta,x)$ is unimodal.

For example, it follows from \cite[Theorem 1.2]{BeckStapledon} that, if $\Delta$ is the $r$\textsuperscript{th} edgewise subdivision of any topological subdivision of a simplex and if $r$ is sufficiently large, then we have $h_j(\Delta_F)\geq h_{|F|-j-1}(\Delta_F)$ for $j<|F|/2$ (see \cite[\S 6]{brunRomer} for a connection between edgewise subdivisions and Veronese subrings). As explained above, this implies the unimodality of $h(\Delta_F,x)-h(\partial (\Delta_F),x)$, and hence the unimodality of the local $h$-polynomial of $\Delta$ by Theorem \ref{thm:formulaoflocalh}.
\end{rem}

Although the polynomials $h(\Delta_F,x)-h(\partial (\Delta_F),x)$ may have negative coefficients in general, the previous remark suggests to study the following problem:

\begin{prob}
 Find classes of subdivisions $\Delta$ such that $h(\Delta_F,x)-h(\partial (\Delta_F),x)$ is nonnegative, unimodal or $\gamma$-nonnegative. Moreover, for those classes try to find a combinatorial interpretation of the coefficients of $h(\Delta_F,x)-h(\partial (\Delta_F),x)$ respectively the coefficients of its $\gamma$-polynomial.
\end{prob}

Possible subdivisions one might consider include chromatic subdivisions (see \cite{Kozlov}), interval subdivisions (see \cite{Walker}) or partial barycentric subdivisions (see \cite{AhmadWelker}). With respect to the second part, one could e.g., study the barycentric subdivision of the cubical barycentric subdivision of the simplex (see \cite[Remark 4.5]{at_survey}) which belongs to the class of subdivisions considered in Theorem  \ref{thm:nonnegative} (this question was raised by Christos Athanasiadis).

In personal communication with Christos Athanasiadis, he asked us if Theorem \ref{thm:nonnegative} can be extended to relative local $\gamma$-vectors (see \cite{at_survey} for the precise definition). We include this question as a reference for future research.

\begin{ques}\label{qu:relative}
Given a CW-regular subdivision $\Gamma$ of a simplex and $E\in \Gamma$, is the relative local $\gamma$-vector of the barycentric subdivision $\sd(\Gamma)$ of $\Gamma$ at $E$ nonnegative?
\end{ques}

\section*{Acknowledgment}
 We wish to thank Christos Athanasiadis for his useful comments. The third author wants to express his gratitude to Isabella Novik for her hospitality at the University of Washington and interesting discussions about the subject.

\bibliographystyle{plain}
\bibliography{local_h_lit}

\end{document}